\documentclass{amsart}

\usepackage{amssymb,amsmath,latexsym,amsthm}
\usepackage[english]{babel}

\newtheorem{theorem}{Theorem}[section]
\newtheorem{lemma}{Lemma}[section]
\newtheorem{proposition}{Proposition}[section]

\theoremstyle{definition}
\newtheorem{definition}{Definition}[section]
\newtheorem{remark}{Remark}[section]

\numberwithin{equation}{section}

\newcommand{\Z}{\mathbb{Z}}
\newcommand{\Q}{\mathbb{Q}}
\newcommand{\R}{\mathbb{R}}
\newcommand{\F}{\mathbb{F}}
\newcommand{\uh}{\mathbb{H}}
\newcommand{\C}{\mathbb{C}}
\newcommand{\ov}{\overline}
\newcommand{\de}{\Delta}
\newcommand{\s}{\sigma}
\newcommand{\p}{\mathfrak{p}}
\newcommand{\qu}{\mathfrak{q}}
\newcommand{\eps}{\varepsilon}
\newcommand{\mco}{\mathcal{O}}
\newcommand{\ok}{\mathcal{O}_K}

\DeclareMathOperator{\im}{Im}
\DeclareMathOperator{\re}{Re}
\DeclareMathOperator{\order}{ord}
\DeclareMathOperator{\maks}{max}
\DeclareMathOperator{\minn}{min}

\DeclareMathOperator{\modulo}{mod}

\begin{document}

\title[A Gap in the Spectrum of the Faltings Height]{A Gap in the Spectrum of the Faltings Height}


\author{Steffen L\"obrich}
\address{Steffen L\"obrich\\
Mathematical Institute, University of Cologne, Weyertal 86-90, 50931 Cologne \\
Germany}
\email{steffen.loebrich@uni-koeln.de}
\urladdr{http://www.mi.uni-koeln.de/$\mathtt{\sim}$sloebric/}



\maketitle

\begin{abstract}
We show that the minimum $h_{\text{min}}$ of the stable Faltings height on elliptic curves found by Deligne is followed by a gap. This means that there is a constant $C >0$ such that for every elliptic curve $E/K$ with everywhere semistable reduction over a number field $K$, we either have $h(E/K)=h_{\text{min}}$ or $h(E/K)\geq h_{\text{min}} +C$. We determine such an absolute constant explicitly. On the contrary, we show that there is no such gap for elliptic curves with unstable reduction.
\end{abstract}

\bigskip
\section{Introduction and Statement of Results}

The Faltings height was introduced by Faltings in his famous proof of the Mordell conjecture \cite{fal} and gives a notion of arithmetic complexity for abelian varieties over number fields. In this article we will only consider elliptic curves, the abelian varieties of dimension $1$. The spectrum of values of the Faltings height $h(E/K)$ for $E$ an elliptic curve over a number field $K$ (see Definition \ref{fhdef}) was first examined by Deligne \cite{del}, who showed that it attains a minimum precisely at elliptic curves with $j$-invariant $0$ and everywhere good reduction. Deligne also gave an explicit expression for the minimal value using the Chowla-Selberg formula. Throughout this article we follow Deligne's normalization.

\begin{theorem}[Deligne]
The Faltings height for elliptic curves is minimal precisely at elliptic curves with $j$-invariant $0$ having everywhere good reduction. The minimum is given by 
\begin{equation*}
h_{\text{min}} = -\frac{1}{2}\log\left(\frac{1}{\sqrt{3}}\left(\frac{\Gamma (1/3)}{\Gamma (2/3)} \right)^3\right)=-0.74875248\dots.
\end{equation*}
\end{theorem}

The minimal value of the Faltings height arises from elliptic curves with complex multiplication by the ring of integers of $\Q (\sqrt{-3})$. The Chowla-Selberg formula and its generalization by Nakkajima and Taguchi \cite{nt} express the Faltings heights of CM-curves in terms of finite sums involving the Gamma-function. However, it is only in the CM-case that such explicit expressions for the Faltings height are known.  \\

Zhang showed that heights on arithmetic varieties induced by hermitian line bundles with smooth metric have isolated minima (see Corollary (5.7) of \cite{zh}). It is therefore natural to ask if the minimum of the Faltings height on elliptic curves is also isolated. If so, one can try to determine an explicit gap and look for a second minimum. The stable Faltings height $h_{\text{stab}}$ (defined in Section 2) is obtained by extending the number field such that the curve has everywhere semistable reduction. We show that, similar to the Weil height (see Definition \ref{whdef}), the values of the Faltings height can get arbitrarily close to $h_{\text{min}}$, while for the stable Faltings height there is a gap behind $h_{\text{min}}$. More precisely, the main result of this article is 
\begin{theorem}\label{gap}
There is a $C>0$ such that for every elliptic curve $E/K$ we have
\begin{equation*}
h_{\text{stab}}(E/K)=h_{\text{min}}\qquad \text{or}\qquad h_{\text{stab}}(E/K) \geq h_{\text{min}} + C ,
\end{equation*}
or equivalently, for every elliptic curve $E/K$ with $j$-invariant not equal to $0$ we have
\begin{equation*}
h(E/K) \geq h_{\text{min}} + C .
\end{equation*}
Moreover, we can choose $C =4.601 \cdot 10^{-18}$.
\end{theorem}
For the proof we relate the Faltings height to the modular height (see Definition \ref{modu}) of an elliptic curve. Silverman estimated the Faltings height from below and above by the modular height \cite{cs}. In Section 3 we will employ estimates for the $j$-function by Faisant and Philibert \cite{fp} in order to determine an absolute constant for the estimate from below. This will be a key ingredient for the full proof given in Section 4. To study the growth of the Faltings height, we need a result by Masser \cite{mas} on the vanishing of the non-holomorphic Eisenstein series of weight $2$.\\

On the other hand, we have
\begin{theorem}\label{nogap}
For every $\eps >0$ there is a number field $K$ and an elliptic curve $E/K$ with $j$-invariant $0$ such that
\begin{equation*}
h_{\text{min}} < h(E/K) < h_{\text{min}} + \eps.
\end{equation*}
\end{theorem}
We show this in Section 5 by constructing elliptic curves with $j$-invariant $0$ over an appropriate sequence of number fields. These curves have bad reduction at only one prime ideal lying over a totally ramified prime number that can be kept small.\\

The smallest value of the stable Faltings height on elliptic curves apart from $h_{\text{min}}$ that we could find by numerically testing elliptic curves with roots of unity and small Salem numbers as $j$-invariants in SAGE \cite{sage} is $h_1 = -0.74862817\dots$, attained at curves with $j$-invariant $1$ and everywhere good reduction. However, we could not prove that the stable Faltings height attains a second minimum.

\section*{Acknowledgements}
This article is based on the author's master's thesis of the same title submitted in December 2014 at TU-Darmstadt. We are deeply indebted to Prof.~Dr.~Philipp Habegger for supervising the thesis and giving much supportive advice in writing this paper. The possibility of a height gap emerged in a conversation between him and Dr.~Ariyan Javanpeykar at the CIRM in 2014. We would also like to thank Prof.~Dr.~Yuri Bilu, Prof.~Dr.~Lenny Fukshansky, Dr.~Ariyan Javanpeykar and Dr.~Lukas Pottmeyer for many helpful remarks.

\section{Definitions and Preliminaries}

In this section we define the modular height and the Faltings height of an elliptic curve. Throughout we denote by $K$ a number field and by $\ok$ its ring of integers. We write $E/K$ for an elliptic curve over $K$ and $j_E$ for its $j$-invariant. We denote by $\tau =x+iy$ a variable in the complex upper half-plane $\uh$ and write
$q=q(\tau)=e^{2\pi i \tau}$, $\de(\tau)=(2\pi)^{12}q\prod_{n\geq 1}(1-q^n)^{24}$ for the \emph{modular discriminant} and $j(\tau)=\frac{1}{q}+744+\dots$ for the \emph{modular $j$-function}. We also define the \emph{closed fundamental domain}
\begin{equation*}
\ov{\mathcal{F}}:=\lbrace \tau\in\uh : -\tfrac{1}{2} \leq\re\tau \leq \tfrac{1}{2} \text{ and } |\tau| \geq 1\rbrace.
\end{equation*} 

For every prime number $p$, we define a $p$-adic absolute value $|\cdot |_p$ on $\Q$ by $|x|_p := p^{-\order _p x}$ for $x\in \Q ^*$ and $|0|_p :=0$. A \emph{place} of a number field $K$ is an equivalence class of non-trivial absolute values on $K$. Let $M_K ^0$ denote the set of all non-archimedean places of $K$. It is well known that every non-archimedean place on $K$ uniquely corresponds to a non-zero prime ideal in $\ok$, so we will freely identify these two sets. Every non-archimedean place $\p\in M_K ^0$ restricts to a non-archimedean place on $\Q$ corresponding to some prime number $p$. We define $|\cdot |_\p$ to be the absolute value in $\p$ that restricts to $|\cdot|_p$. Let $K_\p$ denote the completion of $K$ with respect to the metric defined by $\p$. The \emph{local degree} of $\p$ is defined as $n_\p := [K_\p : \Q _p]$.

\begin{definition}\label{whdef}
The \emph{(absolute logarithmic) Weil height} of an algebraic number $\alpha\in K$ is defined as
\begin{equation*}
h(\alpha):=\frac{1}{[K:\Q]}\left(\sum_{\p\in M_K ^0}n_\p\log \maks \{1,|\alpha|_\p\}+\sum_{\s :K\hookrightarrow \C}\log \maks \{1,|\s (\alpha)|\}\right),
\end{equation*}
where the second sum runs over all embeddings $\s : K\hookrightarrow \C$.
\end{definition}
The product formula implies that the Weil height does not depend on the number field $K$ containing $\alpha$. If $\alpha\in \Q^*$ and $\alpha =\frac{s}{t}$ for coprime integers $s,t$, then we simply have
\begin{equation*}
h(\alpha)=\log\maks \{|s|, |t| \}.
\end{equation*}

\begin{definition}\label{modu}
The \emph{modular height} of an elliptic curve over a number field is the absolute logarithmic Weil height of its $j$-invariant.
\end{definition}

Now we define the Faltings height of an elliptic curve $E/K$. Let $\de_{E/K}$ denote the \emph{minimal discriminant} of $E/K$. For every embedding $\s: K\hookrightarrow \C$, we write $j_\s$ for $\s(j_E)$ and choose a $\tau _\s \in\uh$ such that $j(\tau _{\s})=j_\s$.
\begin{definition}\label{fhdef}
The \emph{Faltings height of $E/K$} is given by
\begin{equation*}
\text{\begin{small}$
h(E/K) := \frac{1}{12[K:\Q]}\left(\log | N_{K/\Q}(\de_{E/K})| - \sum_{\s: K\hookrightarrow \C}\log(| \de(\tau_{\s})| \im (\tau_{\s})^6)\right)+\frac{1}{2}\log\pi
$\end{small}}.
\end{equation*} 
\end{definition}

\begin{remark}
The $\tfrac{1}{2}\log\pi$-term has conventional reasons, as we follow the definition of Deligne \cite{del}. There are several normalizations of the Faltings height going around, all of them differing only by an additive constant. Faltings's original definition \cite{fal} does not have the $\tfrac{1}{2}\log\pi$-term. Silverman's definition \cite{cs} differs from Faltings's by $-2\log 2\pi$, because he defines the $\de$-function with a prefactor of $(2\pi)^{-12}$ instead of $(2\pi)^{12}$. Of course, our results on the existence and size of a gap are independent of the chosen normalization.
\end{remark}

The definition is motivated by Arakelov theory. For an abelian variety $A/K$ with N\'eron model $\mathcal{A}/\ok$, every embedding $\s :K\hookrightarrow \C$ induces a norm on the sheaf of N\'eron differentials $\omega_{\mathcal{A}/\ok}$. These norms make $\omega_{\mathcal{A}/\ok}$ into a metrized line bundle on $\ok$. The Faltings height of $A/K$ is originally defined to be the normalized Arakelov degree of $\omega_{\mathcal{A}/\ok}$ (see \cite{fal}, \S 3). In case of elliptic curves this is equivalent to Definition \ref{fhdef} (see \cite{cs}, Proposition 1.1.).\\

The Faltings height of $E/K$ depends on the field $K$. However, it is the same for every field over which $E$ has everywhere semistable reduction. To get rid of the dependence, we define the \emph{stable Faltings height $h_{\text{stab}}$} by choosing a finite field extension $L/K$ such that $E/L$ has everywhere semistable reduction and setting $h_{\text{stab}}(E/K):= h(E/L)$. The stable Faltings height does not depend on the field $K$ and the chosen field extension. We have $h_{\text{stab}}(E/K)\leq h(E/K)$ for every elliptic curve $E/K$ with equality if and only if $E/K$ has everywhere semistable reduction.\\

Sometimes we want to split the Faltings height into an archimedean and a non-archimedean part. We set 
\begin{equation*}
h_0 (E/K) := \frac{1}{12[K:\Q]}\log | N_{K/\Q}(\de_{E/K})|
\end{equation*}
and
\begin{equation*}
h_{\infty} (E/K) := -\frac{1}{12[K:\Q]}\sum_{\s: K\hookrightarrow \C}\log(| \de(\tau_{\s})| \im (\tau_{\s})^6),
\end{equation*}
so that 
\begin{equation*}
h(E/K)=h_0 (E/K)+h_{\infty} (E/K)+\tfrac{1}{2}\log\pi.
\end{equation*}
Note that $h_0$ is non-negative. The next proposition tells us how the Faltings height behaves under field extensions.

\begin{proposition}\label{p3.1}
Let $E/K$ be an elliptic curve and $L/K$ a finite field extension. Then we have
\begin{itemize}
\item[(i)] $h_{\infty}(E/L)=h_{\infty}(E/K)$
\item[(ii)] $h_{0}(E/L)\leq h_{0}(E/K)$ and $h_{0}(E/L)=h_{0}(E/K)$ if $E/K$ has everywhere  semistable reduction.
\item[(iii)] $h_{0}(E/K)=0$ if and only if $E/K$ has everywhere good reduction.
\end{itemize}
\end{proposition} 

\begin{proof}
\begin{itemize}
\item[(i)]
For every embedding $\s :K\hookrightarrow \C$ there are $[L:K]$ embeddings of $L$ that restrict to $\s$ on $K$. For every such embedding $\rho :L\hookrightarrow \C$, we have $\rho(j_E)=\s(j_E)$, so we may assume that $\tau_{\rho}=\tau_\s$. Hence
\begin{equation*}
\log(| \de(\tau_{\s})| \im (\tau_{\s})^6)=\frac{1}{[L:K]}\sum_{\substack{\rho :L\hookrightarrow \C \\ \rho\mid\s}}\log(| \de(\tau_{\rho})| \im (\tau_{\rho})^6).
\end{equation*}
Summing over all $\s$ and averaging implies the statement.
\item[(ii)]
The ideal $\de _{E/L}$ divides $\de _{E/K}\mco_L$ with equality if $E/K$ has everywhere semistable reduction. Thus 
\begin{equation*}
N_{L/\Q}(\de_{E/L}) \leq N_{L/\Q}(\de_{E/K}\mco_L) =N_{K/\Q}(\de_{E/K})^{[L:K]}
\end{equation*} 
with equality if $E/K$ has everywhere semistable reduction.
\item[(iii)]
$E/K$ has everywhere good reduction if and only if $\de _{E/K}=\ok$ if and only if $h_0 (E/K)=0$.
\end{itemize}
\end{proof}
In particular, if $E/K$ has everywhere potential good reduction, then $h_{\text{stab}}(E/K)=h_{\infty}(E/K)+\tfrac{1}{2}\log\pi$. A well-known result from the theory of elliptic curves states that this is the case if and only if $j_E$ is an algebraic integer.

\section{An Estimate between Faltings Height and Modular Height}

In this section we estimate the Faltings height explicitly from below against the modular height. We factor the principal ideal $(j_E)$ as $(j_E)=\mathfrak{A}\mathfrak{D}^{-1}$ for coprime integral ideals $\mathfrak{A}, \mathfrak{D}\subset \ok$. Then $\mathfrak{D}$ divides $\de_{E/K}$ with equality if and only if $E/K$ has everywhere semistable reduction. Hence we can write $\de _{E/K}=\mathfrak{D}\gamma_{E/K}$ for some integral ideal $\gamma_{E/K}$, such that $\gamma_{E/K}=\ok$ if and only if $E/K$ has everywhere semistable reduction. The ideal $\gamma_{E/K}$ is called the \emph{unstable discriminant of $E/K$}. Silverman proved the following result.

\begin{proposition}[\cite{cs}, Proposition 2.1]\label{p5.1}
There are constants $C_1 , C_2$ such that for every elliptic curve $E/K$ we have
\begin{equation}\label{e5.2}
\begin{split}
C_1&\leq \tfrac{1}{12}h(j_E) + \tfrac{1}{12[K:\Q]}\log |N_{K/\Q} (\gamma_{E/K})| - h(E/K) \\
&\leq \tfrac{1}{2}\log (1+h(j_E)) +C_2.
\end{split}
\end{equation}
In particular, if $E/K$ has everywhere semistable reduction, there is a constant $C$ with
\begin{equation*}
\left| \tfrac{1}{12}h(j_E) -h(E/K)\right|\leq  \tfrac{1}{2}\log (1+h(j_E))+C.
\end{equation*}
\end{proposition}
Gaudron and R\'emond showed that one can choose $C_1=0.72$ (cf.~\cite{gr}, Lemme 7.9). They proved the result for curves with everywhere semistable reduction, but it easily extends to the unstable case. We want determine an absolute $C_2$ in \eqref{e5.2} in order to find a lower bound for the Faltings height in terms of the modular height. In fact, we show that one can choose $C_2=2.071$.
\begin{proposition}\label{p5.4}
Let $E/K$ be an elliptic curve over a number field $K$ with $j$-invariant $j_E$. Then 
\begin{equation*}
h(E/K)> \tfrac{1}{12}h(j_E) -\tfrac{1}{2}\log(1+h(j_E))+ \tfrac{1}{12[K:\Q]}\log |N_{K/\Q} (\gamma_{E/K})|-2.071.
\end{equation*}
\end{proposition}

Before we prove Proposition \ref{p5.4} we need some preparation. To treat the archimedean part we apply estimates for the $j$-function by Faisant and Philibert.
\begin{lemma}[\cite{fp}, \S 2 Lemme 1]\label{l5.2}
\begin{itemize}
\item[(i)] For every $\tau\in\uh$ we have $|j(\tau)|\leq j(i\im\tau)$.
\item[(ii)] For every $y\geq 1$ we have $j(iy)\leq e^{2\pi y} +1193$.
\item[(iii)] For every $\tau\in\ov{\mathcal{F}}$ we have $\im\tau\leq \tfrac{3}{2} \log \maks \lbrace e, | j(\tau)|\rbrace$.
\end{itemize}
\end{lemma}
Now we establish a bound for $|j(\tau)|$ in terms of $\im\tau$.

\begin{lemma}\label{l5.3}
For every $\tau\in\ov{\mathcal{F}}$ we have
\begin{equation*}
\log \maks \{ 1, | j(\tau)|\} \leq 2\pi\im\tau + \log 1193 < 2\pi\im\tau + 7.09.
\end{equation*}
\end{lemma}
\begin{proof}
First we show that 
\begin{equation*}
|j(\tau)|\leq e^{4\pi / \sqrt{3}}+1193<2609
\end{equation*}
for every $\tau\in\ov{\mathcal{F}}$ with $\im\tau\leq 1$.
Since $j$ is a modular function, we have $j(iy)=j(-\tfrac{1}{iy})=j(\tfrac{i}{y})$. It follows 
\begin{equation*}
|j(\tau)|\leq j(i\im\tau)=j\left(\tfrac{i}{\im\tau}\right)\leq e^{2\pi /\im\tau } +1193\leq e^{4\pi /\sqrt{3}}+1193
\end{equation*}
by Lemma \ref{l5.2} (i) and because of $\im\tau \geq \tfrac{\sqrt{3}}{2}$ for $\tau \in\ov{\mathcal{F}}$. Next observe that for $s,t\in \R$ with $2\leq s\leq t$ it holds that
\begin{equation*}
\log (s+t)\leq \log (2t)=\log 2+\log t\leq \log s + \log t,
\end{equation*}
so by Lemma \ref{l5.2} (i) and (ii) we have for $\im\tau\geq 1$
\begin{equation*}
\log |j(\tau)| \leq \log \left(e^{2\pi \im\tau} +1193\right)\leq 2\pi\im\tau +\log 1193.
\end{equation*}
Altogether this implies
\begin{equation*}
\log \maks \{ 1, | j(\tau)|\} \leq \begin{cases}
2\pi\im\tau + \log 1193 &\text{if $\im\tau\geq 1$}\\
\log (e^{4\pi / \sqrt{3}}+1193) &\text{if $\im\tau\leq 1$},
\end{cases}
\end{equation*}
for every $\tau\in\ov{\mathcal{F}}$. Since $\im\tau\geq\tfrac{\sqrt{3}}{2}$ and 
\begin{equation*}
\log (e^{4\pi / \sqrt{3}}+1193) < \pi\sqrt{3}+\log 1193
\end{equation*} 
we have in any case
\begin{equation*}
\log \maks \{ 1, | j(\tau)|\} \leq 2\pi\im\tau + \log 1193 < 2\pi\im\tau + 7.09.
\end{equation*}
\end{proof}

We also need the following inequality by Silverman.
\begin{lemma}[\cite{cs}, \S 2 Exercise 1]\label{l5.4}
For every $\tau\in\ov{\mathcal{F}}$ we have 
\begin{equation*}
\log| \de (\tau)| < -2\pi\im\tau + 22.16.
\end{equation*}
\end{lemma}
\begin{proof}
Since $|q|\leq e^{-\sqrt{3}\pi}$ and $\log (1+x) \leq x$ for all $x\geq 0$, we have
\begin{equation*}
\begin{split}
\sum_{n\geq 1}\log | 1-q^n| &\leq \sum_{n\geq 1}\log (1+ | q| ^n)\leq \sum_{n\geq 1}\log (1+e^{-\sqrt{3}\pi n})\\
&\leq \sum_{n\geq 1}e^{-\sqrt{3}\pi n} = \tfrac{1}{1-e^{-\sqrt{3}\pi}}-1.
\end{split}
\end{equation*}
We deduce that 
\begin{align*}
\log| \de (\tau)| &= \log| q| +24\sum_{n\geq 1}\log | 1-q^n|+12\log(2\pi) \\
&\leq -2\pi\im\tau + 24\left(\tfrac{1}{1-e^{-\sqrt{3}\pi}}-1\right)+12\log(2\pi)\\
&< -2\pi\im\tau + 22.16.\qedhere
\end{align*}
\end{proof}

\begin{proof}[Proof of Proposition \ref{p5.4}]
First we estimate the non-archimedean part $h_0 (E/K)$. 
Note that for every prime ideal $\mathfrak{p}\subset\ok$ we have
\begin{equation*}
\order _{\p}(j_E)= \begin{cases}
\quad\order _{\p}(\mathfrak{A})        &\text{if $\order _{\p}(j_E)\geq 0$}\\
-\order _{\p}(\mathfrak{D})       &\text{if $\order _{\p}(j_E)\leq 0$}\end{cases},
\end{equation*}
where $(j_E)=\mathfrak{A}\mathfrak{D}^{-1}$ as above, and thus
\begin{equation}\label{e5.1}
\maks\{1, | j_E |_{\p} \} = p_\p ^{-\minn\{ 0, \order_{ \p}(j_E)\} / e_\p} = p_\p ^{ \order_{ \p}(\mathfrak{D}) / e_\p}, 
\end{equation}
where $p_\p$ is the unique prime number divided by $\p$ and $e_\p$ the ramification index of $\p$ in $K/\Q$. 
It follows that
\begin{equation}\label{e5.15}
\begin{split}
\sum_{\p\in M_K^ {0}}n_\p \log\maks \{1, | j_E |_{\p} \} &\overset{\eqref{e5.1}}{=}\log\left(\prod_{\p\in M_K^ {0}}p_{\p}^{n_\p \order_{\p}(\mathfrak{D}) / e_\p}\right)=\log |N_{K/\Q}(\mathfrak{D})|\\
&= \log | N_{K/\Q}(\mathfrak{D}\gamma_{E/K})|-\log |N_{K/\Q} (\gamma_{E/K})|\\
&= \log | N_{K/\Q}(\de_{E/K})|-\log |N_{K/\Q} (\gamma_{E/K})|\\
&=12[K:\Q] h_0 (E/K)-\log |N_{K/\Q} (\gamma_{E/K})|.
\end{split}
\end{equation}
Thus  $12 h_0 (E/K)$ is greater than or equal to the non-archimedean part of the modular height of $E/K$ with equality if and only if $E/K$ has everywhere semistable reduction. \\

Putting the inequalities from Lemma \ref{l5.3} and \ref{l5.4} together, we obtain
\begin{equation}\label{e5.35}
-\log|\de (\tau)| > \log \maks \{ 1, | j(\tau)|\} -29.25.
\end{equation}
Moreover, Lemma \ref{l5.2} (iii) yields
\begin{equation}\label{e5.36}
\begin{split}
\log\im\tau &\leq \log\log\maks \{ e, | j(\tau)|\} +\log \tfrac{3}{2}\\
& <\log\log\maks \{ e, | j(\tau)|\} + 0.41
\end{split}
\end{equation}
Now we can estimate the Faltings height. Setting $d:=[K:\Q]$, we obtain
\begin{equation}\label{e5.3}
\begin{split}
12d h(E/K) &= 12dh_0 (E/K) - \sum_{\s: K\hookrightarrow \C} \log(| \de(\tau_{\s})| \im (\tau_{\s})^6)+6d\log\pi\\
&\overset{\eqref{e5.15}}{\geq} \sum_{\p\in M_{K}^0} n_\p \log \maks \lbrace 1, | j_E |_{\p}\rbrace +\log |N_{K/\Q} (\gamma_{E/K})|\\
&\qquad - \sum_{\s: K\hookrightarrow \C}\log| \de(\tau_{\s})| 
- 6\sum_{\s: K\hookrightarrow \C}\log\im\tau_{\s} +6d\log\pi\\
&\overset{\substack{\eqref{e5.35}\\ \eqref{e5.36}}}{>} \sum_{\p\in M_{K}^0}  n_\p\log \maks \lbrace 1, | j_E|_{\p}\rbrace + \sum_{\s: K\hookrightarrow \C}\bigl(\log \max \lbrace 1, | j_E|_{\s}\rbrace -29.25\bigr)  \\
&\qquad-6 \sum_{\s: K\hookrightarrow \C}\bigl(\log\log\maks \lbrace e, | j_E|_\s\rbrace +0.41\bigr)\\
&\qquad + 6d\log\pi +\log |N_{K/\Q} (\gamma_{E/K})|\\
&>d(h(j_E) -24.85) -6 \sum_{\s: K\hookrightarrow \C}\log\log\maks \lbrace e, | j_E|_\s\rbrace \\
&\qquad +\log |N_{K/\Q} (\gamma_{E/K})|.
\end{split}
\end{equation}
The geometric-arithmetic mean inequality yields
\begin{equation}\label{e5.4}
\begin{split}
\sum_{\s: K\hookrightarrow \C}\log\log\maks \lbrace e, | j_E|_\s \rbrace
&= \log\left( \prod_{\s: K\hookrightarrow \C}\log\maks \lbrace e, | j_E|_\s\rbrace\right)\\
&\leq \log \left(\tfrac{1}{d}\sum_{\s: K\hookrightarrow \C}\log\maks \lbrace e, | j_E|_\s\rbrace \right)^d\\
&\leq d\log\left(1+\tfrac{1}{d} \sum_{\s: K\hookrightarrow \C}\log\maks \lbrace 1, | j_E|_\s\rbrace\right)\\
&\leq d\log (1+h(j_E)).
\end{split}
\end{equation}
Now Proposition \ref{p5.4} follows by plugging (\ref{e5.4}) into (\ref{e5.3}).
\end{proof}

\section{Proof of Theorem \ref{gap}}

Now we want to prove Theorem \ref{gap}. Assume throughout this section that $j_E\neq 0$. Let $\varrho =\tfrac{-1+\sqrt{3}i}{2}$ denote the third root of unity in $\uh$. For $\tau\in\uh$ we define
\begin{equation*}
V(\tau) := -\tfrac{1}{12}\log(|\de(\tau)|\im(\tau)^6),
\end{equation*}
so that 
\begin{equation*}
h_{\infty}(E/K)=\frac{1}{[K:\Q]}\sum_{\s: K\hookrightarrow \C} V(\tau_{\s}).
\end{equation*}
Also note that $V(\varrho)=h_{\text{min}}+\tfrac{1}{2}\log\pi$. Since $\de$ is a modular cusp form of weight $12$, the function $V$ is invariant under the  action of $\operatorname{SL}_2 (\Z)$ on $\uh$ and goes to infinity for $\im\tau\rightarrow \infty$. Therefore we may assume that all the $\{\tau_\s \}_{\s :K\hookrightarrow\C}$ lie in $\ov{\mathcal{F}}$.\\

The idea is that if $\tau\in\ov{\mathcal{F}}\backslash \{\varrho , -\varrho^2 \}$ is close to $\varrho$ or $-\varrho^2$, then $|j(\tau)|^{-1}$ becomes large. Consequently, $h(j(\tau)^{-1})=h(j(\tau))$ becomes large. If this happens for too many $\{\tau_\s \}_{\s :K\hookrightarrow\C}$, then $E/K$ has a large modular height. Because of Proposition \ref{p5.4} the curve $E/K$ cannot have a too small Faltings height. If on the other hand a certain part of the $\{\tau_\s \}_{\s :K\hookrightarrow\C}$ lies far from both $\varrho$ and $-\varrho^2$, then the corresponding $\{V(\tau_\s)\}_{\s :K\hookrightarrow\C}$ become too large for $E/K$ to have a small Faltings height. We will now make this idea rigorous and divide the proof into several lemmas.     

\begin{remark}\label{r6.1}
The function
\begin{equation*}
x\mapsto \tfrac{1}{12}x -\tfrac{1}{2}\log(1+x) 
\end{equation*}
is monotonically increasing for $x\geq 5$ and becomes greater than $1.323$ for $x\geq 37.84$.
Thus if $h(j_E)\geq 37.84$ then 
\begin{equation*}
h(E/K) > \tfrac{1}{12}h(j_E) -\tfrac{1}{2}\log(1+h(j_E)) - 2.071 \geq -0.748 > h_{\text{min}}+0.0007 
\end{equation*}
by Proposition \ref{p5.4}.
\end{remark}

\begin{lemma} \label{l6.2}
Let $P\in (0,1)$ and assume that $j_E \neq 0$ and that at least a fraction of $P$ of the conjugates $\{j_{\s}\}_{\s:K\hookrightarrow \C}$ of $j_E$ satisfy 
\begin{equation*}
|j_{\s}|\leq e^{ - 37.84 / P}:=\eps (P),
\end{equation*}
meaning that 
\begin{equation*}
|\{\s : K\hookrightarrow \C : |j_\s | \leq \eps (P)\}| \geq P\cdot [K:\Q].
\end{equation*} 
Then we have $h(j_E)\geq 37.84$ and therefore $h(E/K)\geq h_{\text{min}} + 0.0007$ by Remark \ref{r6.1}. 
\end{lemma}

For the next lemmas we define 
\begin{equation*}
B_{\delta}:=\{\tau\in\ov{\mathcal{F}}:\, |\tau - \varrho|\leq \delta\text{ or } |\tau +\varrho^2 |\leq \delta  \}.
\end{equation*}

\begin{lemma}\label{l6.3}
Let $0\leq \eps \leq 5.08\cdot 10^{-5}$ and $\delta = 0.027\cdot\sqrt[3]{\eps}$. Then for every $\tau\in B_{\delta}$ we have $| j(\tau)|\leq \eps$.
\end{lemma}

\begin{lemma}\label{l6.4}
Let $\delta \in (0, \tfrac{1}{2})$, $q:=-q(\varrho)=e^{-\pi\sqrt{3}}$ and define 
\begin{equation*}
\delta ' := \tfrac{1}{2}\delta\sqrt{1-\tfrac{\delta ^2}{4}} -\tfrac{\sqrt{3}}{4}\delta ^2  
\end{equation*}
and
\begin{equation*}
\begin{split}
C(\delta ')&:= \frac{\pi}{6}\delta' -\frac{1}{2}\log \left(1+\frac{2}{\sqrt{3}}\delta'\right) +\frac{2q(1-e^{-2\pi\delta'})}{1+q}+\frac{2q^3(1-e^{-6\pi\delta'})}{1+q^3} \\
&\qquad +\frac{2}{(1-q^2)(1-q^2 e^{-4\pi\delta'})}- \frac{2}{(1-q^2)^2}. \\
\end{split}
\end{equation*} 
Then for every $\tau\in\ov{\mathcal{F}}\backslash B_{\delta}$ we have
\begin{equation*}
V(\tau) \geq V(\varrho) +C(\delta').
\end{equation*}
\end{lemma}

First we prove that Theorem \ref{gap} follows from Lemma \ref{l6.2}-\ref{l6.4}:
\begin{proof}[Proof of Theorem \ref{gap}]
By Lemma \ref{l6.2} we either have $h(E/K)\geq h_{\text{min}} +0.0007$ or at most a fraction of $P$ of the $\{j_{\s}\}_{\s:K\hookrightarrow \C}$ satisfy $|j_{\s}|\leq \eps (P)$. Let 
\begin{equation*}
\delta (P):= 0.027\cdot \sqrt[3]{\eps (P)}. 
\end{equation*}
Note that $\eps (P)\leq e^{-37.84}< 5.08\cdot 10^{-5}$ and $\delta (P) <\tfrac{1}{2}$ for every $P\in (0,1)$. It follows from Lemma \ref{l6.3} that at most a fraction of $P$ of the $\{\tau_{\s}\}_{\s:K\hookrightarrow \C}$ satisfy $\tau_\s \in B_{\delta(P)}$. In this case Lemma \ref{l6.4} implies
\begin{align*}
h(E/K) &\geq h_{\infty}(E/K) +\frac{1}{2}\log\pi= \frac{1}{[K:\Q]}\sum_{\s: K\hookrightarrow \C} V(\tau_{\s})  +\frac{1}{2}\log\pi\\
&= \frac{1}{[K:\Q]}\left(\sum_{\substack{\s: K\hookrightarrow \C \\ \tau_\s \in B_{\delta (P)}}} V(\tau_{\s})+\sum_{\substack{\s: K\hookrightarrow \C \\ \tau_\s \notin B_{\delta (P)}}} V(\tau_{\s})\right) +\frac{1}{2}\log\pi\\
&\geq \frac{1}{[K:\Q]}\left(\sum_{\substack{\s :\, K\hookrightarrow \C \\ \tau_\s \in B_{\delta (P)}}} h_{\text{min}}+\sum_{\substack{\s :\, K\hookrightarrow \C \\ \tau_\s \notin B_{\delta (P)}}} (h_{\text{min}}+C(\delta' (P)))\right)\\
&\geq h_{\text{min}} + (1-P) C(\delta' (P)),
\end{align*}
where
\begin{equation*}
\delta' (P):= \tfrac{1}{2}\delta(P)\sqrt{1-\tfrac{\delta (P)^2}{4}}-\tfrac{\sqrt{3}}{4}\delta (P)^2  
\end{equation*}
and $C(\delta')$ as in Lemma \ref{l6.4}. As $P$ was arbitrary, we obtain a gap of 
\begin{equation*}
\minn \{  0.0007 ,\, \maks_{P\in (0,1)}(1-P) C(\delta' (P)) \}.
\end{equation*}
We will see that $\maks_{P\in (0,1)}(1-P) C(\delta' (P))$ 
is much smaller than $0.0007$, so the gap function is given by
\begin{align*}
P \longmapsto &(1-P) C(\delta' (P))\\
=&(1-P)\Bigl( \frac{\pi}{6}\delta'(P) -\frac{1}{2}\log \left(1+\frac{2}{\sqrt{3}}\delta'(P)\right) +\frac{2q(1-e^{-2\pi\delta'(P)})}{1+q}\\
&\quad +\frac{2q^3(1-e^{-6\pi\delta'(P)})}{1+q^3}+\frac{2}{(1-q^2)(1-q^2 e^{-4\pi\delta'(P)})}- \frac{2}{(1-q^2)^2}\Bigr).
\end{align*}
Numerical computations at high precision in SAGE \cite{sage} suggest that the function has a maximum at $0.964\dots$. So we choose $P=0.964$ and obtain a gap of $4.601 \cdot 10^{-18}$.
\end{proof}

Now we prove Lemma \ref{l6.2}--\ref{l6.4}.
\begin{proof}[Proof of Lemma \ref{l6.2}]
First note that $\eps (P)=e^{-37.84 / P}<1$ for $P\in (0,1)$. Assuming that at least a fraction of $P$ of the $\{j_{\s}\}_{\s:K\hookrightarrow \C}$ satisfy 
$|j_{\s}|\leq \eps(P)$ and that $j_E \neq 0$ we obtain 
\begin{equation*}
\begin{split}
h(j_E ^{-1})&\geq \frac{1}{[K:\Q]}\sum _{\s: K\hookrightarrow \C}\log\maks\{ 1,|j_{\s}|^{-1}\}\\
&\geq \frac{1}{[K:\Q]}\sum _{\substack{\s: K\hookrightarrow \C \\ |j_{\s}|\leq \eps}}\log\maks\{ 1,|j_{\s}|^{-1}\} \\
&=\frac{1}{[K:\Q]}\sum _{\substack{\s :\, K\hookrightarrow \C \\ |j_{\s}|\leq \eps}}-\log |j_{\s}|  \geq -P\log\eps(P) =37.84.
\end{split}
\end{equation*}
It is well-known that $h(\alpha ^{-1})=h(\alpha)$ for every $\alpha\in\ov{\Q}\setminus \{0\}$, which proves Lemma \ref{l6.2}.
\end{proof}

\begin{proof}[Proof of Lemma \ref{l6.3}]
Bilu, Luca and Pizarro-Madariaga showed that
\begin{equation*}
|j(\tau)|\leq 47000\cdot |\tau -\varrho|^3
\end{equation*}
for all $\tau\in\ov{\mathcal{F}}$ with $|\tau -\varrho|\leq 0.001$ (see \cite{bil}, Proposition 2.2). Thus if $\eps\leq 5.08\cdot 10^{-5}$ and $\delta=0.027\cdot \sqrt[3]{\eps}$, then $\delta \leq 0.001$ and $|j(\tau)|\leq \eps$ for all $\tau\in\ov{\mathcal{F}}$ with $|\tau -\varrho|\leq \delta$. The same statement holds for $-\varrho^2$ instead of $\varrho$.
\end{proof}

To prove Lemma \ref{l6.4}, we first show that we can restrict ourselves to the half-line $\re\tau=-\tfrac{1}{2}$:
\begin{lemma}\label{l6.5}
For all $\tau = x+iy\in\ov{\mathcal{F}}$ we have
\begin{equation*}
V(\tau)\geq V(-\tfrac{1}{2}+iy)
\end{equation*}
and $V$ is monotonically increasing on the half line defined by $x=-\tfrac{1}{2}$ and $y\geq\tfrac{\sqrt{3}}{2}$.
\end{lemma}

\begin{proof}
We regard $V$ as a function of the real variables $x,y$ and examine the partial derivatives. Let 
\begin{equation*}
E_2(x,y):=1-24\sum_{n\geq 1}\frac{nq^n}{1-q^n}-\frac{3}{\pi y}
\end{equation*}
denote the non-holomorphic Eisenstein series of weight $2$. Direct calculation shows that
\begin{equation*}
\partial _x V(x,y)-i\partial _y V(x,y)=-\tfrac{\pi i}{6}E_2(x,y).
\end{equation*}
According to Masser's proof of Lemma 3.2 in \cite{mas}, we have $E_2 (x,y)\in\R$ and hence $\partial _x V(x,y)=0$ if and only if $x\in \tfrac{1}{2}\Z$. In the expression
\begin{equation*}
|\de (x,y)|= (2\pi)^{12}|q|\prod_{n\geq 1}|1-q^n|^{24}=(2\pi)^{12}|q| \prod_{n\geq 1}(1-2\re (q^n)+|q|^{2n})^{12}
\end{equation*}
only the $\re (q^n)$-term depends on $x$. For fixed $y>0$ the product is minimal if $q>0$, i.e.~$x\in\Z$. Therefore it has to be maximal for $x\in \tfrac{1}{2}+\Z$.
Hence $V(x,y) = -\tfrac{1}{12}\log(|\de(x,y)|y^6)$ is minimal for $x\in\tfrac{1}{2}+\Z$ if $y$ is fixed. Furthermore, Masser's proof shows that the only zero of $E_2$ on the half-line $x=-\tfrac{1}{2}$ in $\uh$ is $\varrho$. Since $\partial_y V$ is continuous and $\lim_{y\rightarrow\infty}V(x,y)=\infty$, it follows that $\partial_y V(-\tfrac{1}{2},y) > 0$ for $y>\tfrac{\sqrt{3}}{2}$, hence $V$ is monotonically increasing on this half-line.
\end{proof}    

\begin{proof}[Proof of Lemma \ref{l6.4}]
For $\delta\in (0,\tfrac{1}{2})$, the upper intersection point $w$ of the circle $|\tau -\varrho|=\delta$ and the unit circle has imaginary part
\begin{equation*}
\im w = \tfrac{\sqrt{3}}{2}  + \tfrac{1}{2}\delta\sqrt{1-\tfrac{\delta ^2}{4}}-\tfrac{\sqrt{3}}{4}\delta ^2 = \im\varrho + \delta'.
\end{equation*}
Since this is the smallest possible imaginary part on the two arcs constituting $\partial B_{\delta}\cap\ov{\mathcal{F}}$, we have $\im\tau\geq \im w = \im\varrho + \delta'$ for every $\tau\in\ov{\mathcal{F}}\backslash B_{\delta}$. By Lemma \ref{l6.5} we have
\begin{equation*}
V(\tau)\geq V(-\tfrac{1}{2}+i\im\tau)\geq V(\varrho + i\delta ')
\end{equation*}
for every $\tau\in\ov{\mathcal{F}}\backslash B_{\delta}$, so it suffices to show the estimate for $\varrho + i\delta' \in -\tfrac{1}{2}+i\R$ . For the rest of the proof we set $q:=-q(\varrho)=e^{-\pi\sqrt{3}}$ and $t:=e^{-2\pi\delta'}$ so that $q(\varrho + i\delta')=-qt$. First we estimate the difference of the infinite product of $\de (\tau)$.
\begin{align*}
\log\prod_{n\geq 1}&|1-(-qt)^n|-\log\prod_{n\geq 1}|1-(-q)^n|=\sum_{n\geq 1}\log\left|\frac{1-(-qt)^n}{1-(-q)^n}\right|\\
&\leq\sum_{n\geq 1}\left(\frac{1-(-qt)^n}{1-(-q)^n}-1\right)=\sum_{n\geq 1}\frac{(-q)^n(1-t^n)}{1-(-q)^n}\\
&\overset{(*)}{\leq} -\frac{q(1-t)}{1+q}-\frac{q^3(1-t^3)}{1+q^3} + \frac{1}{1-q^2} \left(\sum_{n\geq 1}q^{2n}(1-t^{2n})\right)\\
&=-\frac{q(1-t)}{1+q}-\frac{q^3(1-t^3)}{1+q^3} -\frac{1}{(1-q^2)(1-q^2 t^2)}+ \frac{1}{(1-q^2)^2}
\end{align*}
For $(*)$ we used that since $0<q,t<1$ we have
\begin{equation*}
\frac{(-q)^n(1-t^n)}{1-(-q)^n}<0 
\end{equation*}
for $n$ odd and $0< q^n < q^2$ for $n\geq 2$ even.\\

Hence we have for $\tau\in \ov{\mathcal{F}}\backslash B_{\delta}$:
\begin{align*}
V(\tau)-V(\varrho) &\geq V(\varrho +i\delta') - V(\varrho) = -\frac{1}{12}\log \left|\frac{\de (\varrho +i\delta')}{\de(\varrho)}\right| -\frac{1}{2}\log\left(1+\frac{2\delta'}{\sqrt{3}}\right)\\
&=\frac{\pi}{6}\delta' - 2\sum_{n\geq 1}\log\left\lvert\frac{1-e^{-2\pi n\delta'}(-q)^n}{1-(-q)^n}\right\rvert -\frac{1}{2}\log\left(1+\frac{2\delta'}{\sqrt{3}}\right)\\
&\geq \frac{\pi}{6}\delta' -\frac{1}{2}\log \left(1+\frac{2}{\sqrt{3}}\delta'\right) +\frac{2q(1-e^{-2\pi\delta'})}{1+q}+\frac{2q^3(1-e^{-6\pi\delta'})}{1+q^3} \\
&\qquad +\frac{2}{(1-q^2)(1-q^2 e^{-4\pi\delta'})}- \frac{2}{(1-q^2)^2}=C(\delta').
\end{align*} 
\end{proof}

\section{A Construction for the Case $j_E=0$}

In this section we show that if $E/K$ with $j_E=0$ is not supposed to have everywhere semistable reduction, then $h(E/K)$ can get arbitrarily close to $h_{\text{min}}$.\\

First we recall a well-known result on Eisenstein polynomials. Let $K$ be a number field, $\p\subset \ok$ a prime ideal, $f\in\ok [X]$ a monic Eisenstein polynomial for $\p$ of degree $n$ and $\alpha$ a zero of $f$. Then $f$ is irreducible in $K[X]$ and $\p$ is totally ramified in the extension $K(\alpha)/K$. If in particular $K=\Q$, $\p=(p)$ and furthermore $f(0)=p$, then $p \mco _{\Q(\alpha)}= \alpha ^n \mco _{\Q(\alpha)}$ and $N_{\Q(\alpha)/\Q}(\alpha)=(-1)^n p$.

\begin{lemma}\label{eisen}
Let $n\geq 1$ and $p$ a prime number with $p\equiv (-1)^n \modulo 9$. Then there is a monic polynomial $f$ of the form
\begin{equation}\label{9erform}
f(X)=(X-1)^n + \sum_{k=1}^n 9^k b_k (X-1)^{n-k}
\end{equation}
with $b_k \in\Z$ and $f(0)=p$ that is Eisenstein for $p$. If $\alpha$ is a zero of $f$, then $\tfrac{1}{9}(\alpha -1)$ is an algebraic integer.
\end{lemma}

\begin{proof}
Let $p=(-1)^n +9m$ and $f$ a polynomial of the form \eqref{9erform}. We have to show that there is a vector $b = (b_1, \dots , b_n)^{\top}\in \Z^n$ such that $f(0)=p$ and $f$ is Eisenstein for $p$. Let $a_1, \dots , a_{n} \in \Z$ such that $f(X)=X^n +\sum_{k=1}^{n} a_k X^{n-k}$. We have 
\begin{equation*}
f(0)=a_n=\sum_{k=1}^n (-1)^{n-k} 9^{k}b_k +(-1)^n. 
\end{equation*}
Thus the condition that $f(0)=p$ can be written as
\begin{equation}\label{spi}
\psi (b) := \sum_{k=1}^n (-1)^{n-k} 9^{k-1} b_k=m.
\end{equation}
One integral solution is $u_1 = ((-1)^{n-1}m, 0, \dots ,0)^{\top}$, so the solution set of \eqref{spi} is given by $u_1 + \ker\psi$. We write $\F_p ^n$ for the $n$-dimensional vector space over the field with $p$ elements and a tilde for component-wise reduction $\modulo p$. Obviously $\ker \widetilde{\psi}$ is an $(n-1)$-dimensional subspace of $\F _p ^n$ that contains $\widetilde{\ker\psi}$. Since $\ker\psi$ is a primitive subgroup of $\Z^n$, the subspace $\widetilde{\ker\psi}$ is also $(n-1)$-dimensional. Hence we have $\ker \widetilde{\psi} = \widetilde{\ker\psi}$,
i.e.~every vector $\tilde{v}\in\ker\widetilde{\psi}$ lifts to a vector $v\in\ker\psi$. \\

Let now $a=(a_1,\dots ,a_{n-1})^{\top}\in \Z^{n-1}$. Then $a$ and $b$ are related by the system of linear equations
\begin{equation*}
 \begin{pmatrix}
   9 & 0 & \cdots & & 0 \\
   -9(n-1) & 9^2 & \cdots & & 0 \\
   \vdots  & \vdots  & \ddots & & \vdots  \\
   (-1)^{n} 9(n-1) &  & \cdots & 9^{n-1}& 0
  \end{pmatrix}
  \begin{pmatrix}
  b_1 \\
  b_2 \\
  \vdots \\
  b_n
  \end{pmatrix} 
  + \begin{pmatrix}
  -n \\
  \binom{n}{2} \\
  \vdots \\
  (-1)^{n-1}n
  \end{pmatrix} 
  = \begin{pmatrix}
  a_1 \\
  a_2 \\
  \vdots \\
  a_{n-1}
  \end{pmatrix}
\end{equation*}
which we will write as $Ab+u_2 =a$. 
$A$ is a lower triangle $(n-1) \times n$-matrix with nonzero diagonal entries and therefore has rank $(n-1)$. The coefficients $a_k$ are divisible by $p$ if
\begin{equation*}
\tilde{A}\tilde{b} + \tilde{u}_2 =\tilde{a}=0.
\end{equation*}
Since the diagonal entries of $A$ are not divisible by $p$, the reduced matrix $\tilde{A}$ also has rank $(n-1)$ and we have $\ker\tilde{A}=\{0\} ^{n-1}\times \F_p$. It follows that $\ker\widetilde{\psi}\cap \ker\tilde{A} =\{0 \}$, so $\tilde{A}$ defines a bijection between $\ker\widetilde{\psi}$ and $\F_p ^{n-1}$. Let $\tilde{u} := \tilde{A}\tilde{u}_1 + \tilde{u}_2$. Then the system 
\begin{equation*}
\tilde{A}\tilde{b'} +\tilde{u}=0
\end{equation*} 
has a unique solution $\tilde{b'}\in\ker\widetilde{\psi}$, that lifts to an integral vector $b'\in\ker\psi$. Let $b=u_1 + b'$. Then $\psi (b) =m$ and 
\begin{equation*}
\tilde{A}\tilde{b}+\tilde{u}_2 =  \tilde{A}(\tilde{u}_1+\tilde{b'})+\tilde{u}_2 = \tilde{A}\tilde{b'} + \tilde{u}=0,
\end{equation*} 
so the polynomial defined by $b$ has the desired properties.\\ 

Let eventually $\p$ be a prime in $\mco_{\Q(\alpha)}$ lying over $3$. Then
\begin{equation*}
|\alpha -1|_\p ^n \leq \maks_{1\leq k\leq n} |9^k b_k (\alpha -1)^{n-k}|_\p \leq 9^{-k_0}|\alpha -1|_\p ^{n-k_0}
\end{equation*}
for some $k_0$. Thus $|\alpha -1|_\p \leq 9^{-1}$.
\end{proof}

Now we are ready to construct the appropriate number fields.
\begin{proof}[Proof of Theorem \ref{nogap}]
For an algebraic integer $\alpha\in\ov{\Q}^*$, we consider elliptic curves with $j$-invariant $0$ given by the equations  
\begin{equation}\label{we1}
E_1: Y^2 + \alpha Y = X^3 
\end{equation}
with discriminant $\de_1 = -27\alpha ^4$ and 
\begin{equation}\label{we2}
E_2: Y^2 = X^3 + (\varrho ^2-1)X^2 -\varrho^2 X + \tfrac{i}{3\sqrt{3}}(\alpha ^2 -1),
\end{equation}
where again $\varrho=\tfrac{-1+\sqrt{3}i}{2}$, with discriminant $\de_2 = 16\alpha ^4$. Note that $j_{E_1}=j_{E_2}=0$. One can think of $E_2$ as a disturbed Legendre equation. Let $\zeta_8=e^{-i\pi/8}$, $K=\Q(\sqrt[3]{4} , \sqrt[4]{27}\zeta_8 )$ and $L=K(\alpha)$. Then we have $[K:\Q]=12$ and $i, \varrho, \sqrt{3}\in K$. The discriminant of the extension $K/\Q$ is $2^{16}\cdot 3^{15}$, so $2$ and $3$ are the only primes that ramify in $K/\Q$. The change of variables
\begin{equation*}
X = \tfrac{1}{\sqrt[3]{4}} ((1-\varrho )X' -1), \quad Y = \tfrac{1}{2}(\sqrt[4]{27}\zeta_8 Y' -\alpha)
\end{equation*}
shows that $E_1/L$ and $E_2/L$ are isomorphic. We denote the corresponding curve by $E/L$. Now we assume that $\alpha$ is coprime to $2$ and $3$. Then \eqref{we1} is a minimal Weierstra\ss{} equation at every prime in $\mco _L$ lying over $2$. If $\tfrac{1}{3\sqrt{3}}(\alpha^2 -1)$ is an algebraic integer, then \eqref{we2} is integral and therefore minimal at every prime in $\mco_L $ dividing $3$. In this case, $E/L$ has bad reduction only at primes dividing $\alpha$.\\

Our task is to choose a suitable $\alpha$. Let $n\geq 1$ and $p$ a prime number with $p\equiv (-1)^n \modulo 9$. By Lemma \ref{eisen} there is a monic polynomial $f\in \Z[X]$ of degree $n$ that is Eisenstein for $p$ and of the form \eqref{9erform} with $f(0)=p$. We choose $\alpha$ to be a zero of $f$, so that $\tfrac{1}{9}(\alpha -1)$ and hence also $\tfrac{1}{3\sqrt{3}}(\alpha^2 -1)$ is an algebraic integer. Moreover, $p$ factors as $p\mco _{\Q(\alpha)}=\alpha^n \mco_{\Q(\alpha)}$. We have $p\geq 5$, so $p$ does not ramify in $K/\Q$. Let $p\mco _K=\p _1 \cdots \p_m$ be the prime factorization of $p\mco_K$.
Regarded as a polynomial in $\mco_K [X]$, $f$ is monic and Eisenstein for every $\p_j$ and therefore irreducible in $K[X]$. Thus $[L:K]=n$ and every $\p_j$ is totally ramified in $L/K$, say $\p_j \mco _L =\mathfrak{P}_j ^n$. Now we have 
\begin{equation*}
\alpha^n \mco_L= p\mco_L =  \p _1 \cdots \p_m \mco_L= \mathfrak{P} _1 ^n\cdots \mathfrak{P}_m ^n.
\end{equation*}
It follows that $\alpha\mco_L = \mathfrak{P} _1 \cdots \mathfrak{P}_m $, so $\alpha\mco_L$ is unramified in $L/\Q(\alpha)$. If $\qu\subset\mco _L$ is a prime ideal dividing $\alpha\mco_L$, then
\begin{equation*}
\order_\qu \de_1 =\order_\qu \de_2 = \order_\qu (\alpha^4\mco_L) = 4<12. 
\end{equation*}
It follows that \eqref{we1} and \eqref{we2} are minimal Weierstra\ss{} equations for $E/L$ at $\qu$. Hence 
\begin{equation*}
N_{L/\Q}(\de_{E/L})= N_{L/\Q}(\alpha ^4 ) = N_{\Q(\alpha)/\Q}(\alpha )^{4[L:\Q(\alpha)]}= p^{4[L:\Q(\alpha)]}. 
\end{equation*}
and 
\begin{equation*}
h_0 (E/L) = \frac{1}{12[L:\Q]}\log N_{L/\Q}(\de_{E/L})= \frac{4[L:\Q(\alpha)]}{12[L:\Q]}\log p = \frac{\log p }{3n}.
\end{equation*}  
We can keep $p$ bounded as $n$ grows, for example we may choose $p=17$ for $n$ odd and $p=19$ for $n$ even. Then $h_0 (E/L)$ gets arbitrarily close to $0$. Since $h_{\infty}(E/L)+\frac{1}{2}\log\pi =h_{\text{min}}$ for all elliptic curves $E/L$ with $j$-invariant $0$, this implies Theorem \ref{nogap}.
\end{proof}

\end{document}